\documentclass[12pt]{article}
\usepackage{latexsym}
\usepackage{amsfonts}
\usepackage{amsmath}
\usepackage{mathrsfs}
\usepackage[pdftex]{graphicx}
\usepackage{amssymb}

\def\RR{\mathbb R}
\def\CC{\mathbb C}

\def\ro{\mathring{r}}
\def\CP{\mathbb{CP}}
\def\eea{\end{eqnarray*}}

\DeclareMathOperator{\grad}{grad}
\DeclareMathOperator{\Aut}{Aut}
\DeclareMathOperator{\Hess}{Hess}

\newtheorem{main}{Theorem}

\newtheorem{defn}{Definition}
\newtheorem{thm}{Theorem}
\newtheorem{prop}{Proposition}

\newtheorem{lem}{Lemma}

\newenvironment{proof}{\medskip \noindent
{\bf Proof.}}{\hfill \rule{.5em}{1em}
\\}

\newenvironment{ack}{\mbox{ }\\{\bf  Acknowledgments}\mbox{ }}{
\hfill \mbox{}\bigskip}
\begin{document}
\sloppy

\title{The Einstein-Maxwell Equations, K\"ahler Metrics,  and Hermitian Geometry}

\author{Claude LeBrun\thanks{Supported 
in part by  NSF grant DMS-1205953.} 
\\ 
SUNY Stony
 Brook 
  }

\date{November 11, 2014; revised January 17, 2015}
\maketitle

\hspace{1.8in}
\begin{minipage}{3.3in}
\begin{quote} {\em To my dear friend and esteemed 
 colleague Paul Gauduchon, on the occasion of
his seventieth birthday. }
\end{quote}
\end{minipage}

\bigskip

\begin{abstract}
Any constant-scalar-curvature K\"ahler (cscK) metric on a complex 
surface may be viewed as a solution of the Einstein-Maxwell equations, and this allows one \cite{lebem,yujen} to produce
 solutions of these equations on any  $4$-manifold that arises as a compact complex surface
with $b_1$ even. However,   
  not all solutions of the Einstein-Maxwell equations on 
such manifolds arise in this way;  new  examples can be  constructed 
by means of  conformally  K\"ahler geometry. 
\end{abstract}

~

\hspace{3in} %

Let $M$ be a smooth compact oriented $4$-manifold, equipped with  a Riemannian metric $h$ and 
a  real-valued $2$-form $F$. One then says that the triple $(M,h,F)$ satisfies the 
{\em Einstein-Maxwell equations} if the relations 
\begin{eqnarray}
\label{closed}dF&=&0\\
\label{coclosed}d \star F&=&0\\
\label{energy}\Big[r+ F\circ F\Big]_0
&=&0
\end{eqnarray}
all hold, where $r$ is the Ricci tensor of $g$,
the subscript 
$[~]_0$ indicates
the trace-free part with respect to $g$, 
and  the symmetic tensor 
$(F\circ F)_{jk}= {F_j}^\ell F_{\ell k}$ is obtained by
composing $F$ with itself as an endomorphism of $TM$. 
If $M$ is {\em compact}, equations (\ref{closed}--\ref{coclosed}) can be
unambiguously summarized as saying that $F$ is a {\em harmonic $2$-form},
but  \eqref{energy} is less familiar to differential geometers. 
In physics, these equations represent the interaction of a gravitational field $h$
and an electromagnetic field $F$; physicists would  call these the
``Euclidean Einstein-Maxwell equations with cosmological constant,''
 embellishing the terminology to emphasize that we have 
taken $h$ to be a Riemannian metric rather than a Lorentzian one, and that we 
are (implicitly) allowing the scalar curvature to be an arbitrary constant, rather than requiring it to vanish.

While relativists understand  these equations as the Euler-Lagrange equations of a suitable  Lagrangian, 
it was pointed out in \cite{lebem} that they
 also arise from a non-traditional variational problem 
that, while apparently unfamiliar to physicists,  is of immediate interest in  
 Riemannian geometry. Indeed, suppose that $M^4$ is compact,  let
$[\omega ]\in H^2(M,\RR)$ be a fixed cohomology class with $[\omega ]^2 > 0$, and 
let 
$\mathscr{G}_{[\omega]}$ be the Fr\'echet manifold of  smooth 
Riemannian metrics $g$ on $M$ for which 
the harmonic representative $\omega$ of $[\omega ]$ is self-dual with respect to the relevant metric 
$g$.
Letting $s$ denote the scalar curvature of a Riemannian metric, we
 can then consider the    {\em Einstein-Hilbert functional} 
 \begin{equation}\label{hilbert}
g\longmapsto \frac{\int_M s_g~d\mu_g}{\sqrt{\int_M ~d\mu_g}}
\end{equation}
as a functional on $\mathscr{G}_{[\omega]}$, rather than
 as a functional on the space
of {\em all} Riemannian metrics. A metric $h\in \mathscr{G}_{[\omega]}$ is then
a critical point for this problem if and only if  there is a harmonic $2$-form $F$ with self-dual part 
$F^+\in [\omega]$
such that the pair $(h,F)$ solves (\ref{closed}--\ref{energy}). Similarly, 
the critical points of the Calabi-type functional 
\begin{equation}
\label{calabi}
g\longmapsto {\int_M s_g^2~d\mu_g}
\end{equation}
on $\mathscr{G}_{[\omega]}$
are either Einstein-Maxwell or  scalar-flat. 

It is now worth  emphasizing   that equations (\ref{closed}--\ref{energy}) imply 
that the scalar curvature of $h$ is constant. Indeed, if $M$ is compact, this can be deduced
from the fact that if $h$ belongs to $\mathscr{G}_{[\omega]}$, so does its entire conformal 
class; and the restriction of \eqref{hilbert} to a conformal class is exactly the functional used in 
the Yamabe problem to identify metrics of constant scalar curvature. However, this fact 
about  (\ref{closed}--\ref{energy}) can also be  inferred directly, via a local calculation. Indeed, 
a contraction of the second Bianchi identity tells us that 
$$\nabla\cdot \ro = {\textstyle \frac{1}{4}}ds$$
in dimension $4$, where $\ro$ is the trace-free Ricci tensor, and where $s$ is the scalar curvature. 
On the other hand, \eqref{energy} can be rewritten as 
$$\ro = -2 F^+\circ F^-,$$
where 
$F^\pm = \frac{1}{2} [ F\pm \star F]$
denotes  the self-dual or anti-self-dual part of $F$, depending on the sign. 
But, as we shall see in \S \ref{overture} below, 
\begin{equation}
\label{leibniz}
\nabla\cdot (F^+ \circ F^-) = F^- (\nabla\cdot F^+)  + F^+  (\nabla \cdot F^-),
\end{equation}
and it therefore follows that $ds=0$ if $F^+$ and $F^-$ are both co-closed. The latter stipulation 
is of course exactly equivalent to equations (\ref{closed}--\ref{coclosed}).

The author's main point in  \cite{lebem} was that constant-scalar-curvature K\"ahler (cscK) metrics
on complex surfaces $(M^4,J)$ 
can be considered as solutions of the Einstein-Maxwell equations; moreover, 
when $M$ is compact and the scalar curvature is non-positive,  such solutions 
 are actually {\em minima} of \eqref{calabi} on $\mathscr{G}_{[\omega]}$,  rather than just critical points. 
 
 This article, however, will focus on another class  of solutions,  suggested by 
  a recent paper of  Apostolov, Calderbank and Gauduchon \cite{acg1}. We begin with a 
  definition that is ostensibly much weaker than theirs:

\begin{defn}
Let $J$ be an integrable almost-complex structure on $M$, thus making $(M^4,J)$ into a complex surface. 
We will say that a solution $(h,F)$ of the Einstein-Maxwell equations (\ref{closed}--\ref{energy}) 
on $(M,J)$  is {\em strongly Hermitian} if 
 $h$ and $F$ are  both invariant under the action of $J$:
\begin{eqnarray*}
h&=& h(J\cdot , J\cdot ) , \\
F&=& F ( J\cdot , J \cdot ).
\end{eqnarray*}
\end{defn}

Our first main result, proved in \S \ref{scenario} below,  asserts that, aside from a well-understood
 exceptional case, solutions of this type are in fact 
{\em  globally} of the type studied locally
by Apostolov-Calderbank-Gauduchon \cite{acg1}: 

\begin{main}\label{clef}
Let  $(h,F)$ be a strongly Hermitian solution of the Einstein-Maxwell equations 
 on a  (connected) complex surface $(M^4,J)$. Then either 
$h$ is  Einstein and anti-self-dual, or else there is a $J$-compatible K\"ahler metric $g$ 
on $M$ and  a real holomorphy potential 
$f > 0$ on $(M,J,g)$  such that $h=f^{-2}g$ has constant scalar curvature, and
such that 
$F^+$ is a constant times the 
K\"ahler form $\omega$ of $g$. Conversely, if $(M^4,g,J)$ is a K\"ahler manifold
and if $f>0$ is a real holomorphy potential such that $h=f^{-2}g$ has constant scalar curvature,
then there is a unique harmonic $2$-form $F$ on $M$ with $F^+=\omega$ such that 
$(h,F)$ solves the Einstein-Maxwell equations. 
\end{main}

The above result does not require that $M$ be compact, or  that
$h$ be complete. 
However, the statement can be simplified  in the compact case: 

\begin{main} \label{alto}
Let  $(h,F)$ be a strongly Hermitian solution of the Einstein-Maxwell equations 
on a compact complex surface $(M^4,J)$. Then 
there is a K\"ahler metric $g$ on $(M,J)$, together with a holomorphy potential 
$f> 0$ such that $h=f^{-2}g$, and such that  $F^+$ is a constant times the 
K\"ahler form $\omega$ of $g$.
\end{main}

One consequence  is that any compact complex surface that admits a 
non-K\"ahler, strongly Hermitian solution of the Einstein-Maxwell equations must be 
rational or ruled; cf.\  \S \ref{conclusion} below for details. While it is beyond the scope of the present paper to try 
to classify the rational or ruled
surfaces which do actually admit such solutions, we will
  construct  non-trivial examples in \S \ref{construct}
that demonstrate that such solutions  really do exist  on at least one rational ruled surface:

\begin{main}\label{behold}
Let $[\omega ]$ be a  K\"ahler class  on $(M,J) = \CP_1 \times \CP_1$ 
for which the area of  one factor $\CP_1$ is more than quadruple the area of the other. 
Then $[\omega ]$ contains pairs of K\"ahler metrics which engender  two geometrically distinct solutions
of the Einstein-Maxwell equations 
(\ref{closed}--\ref{energy}) via Theorem \ref{clef}.
\end{main}

In other words, if $[\omega ](\mathfrak{S}_1) > 4 [\omega ](\mathfrak{S}_2)$, where $\mathfrak{S}_1$ and $\mathfrak{S}_2$ are the
homology classes of the two factor $\CP_1$'s,  there are at least two different orbits of the action of 
$$\Aut_0(\CP_1\times \CP_1)= PSL(2, \CC ) \times PSL (2, \CC)$$ 
on  metrics in
the given K\"ahler class  which engender  strongly Hermitian solutions of the 
Einstein-Maxwell equations; moreover,  the Hermitian metrics $h$ arising from these two different 
orbits can  be distinguished from each other by their curvature properties.  
This indicates that the uniqueness theorems 
for cscK metrics on compact complex manifolds 
\cite{xxgang,donaldsonk1} unfortunately do not generalize to 
Einstein-Maxwell metrics in any obvious manner. 

These new examples have some curious incidental 
properties that almost seem like an invitation to  revisit  difficult open questions  
regarding the Yamabe problem on $S^2\times S^2$. A discussion of these issues and 
other open problems can be found  in the concluding  section of this article. 

\section{Preliminaries}
\label{overture}

Let us now  discuss some  basic facts 
needed to  provide a solid foundation for the rest of the article. 
 We begin with a different characterization of solutions of the 
Einstein-Maxwell equations;   cf. \cite[Proposition 5]{acg1}.

\begin{prop}
Let $(M,h)$ be an oriented  Riemannian $4$-manifold, and let $F$ be a real-valued $2$-form on 
$M$. Let $s$ and $\ro$ respectively denote the scalar curvature and trace-free Ricci curvature of $h$. 
Let $Y\subset M$ be the (possibly empty) open set where $F^+\neq 0$. Then $(h,F)$ solves the Einstein-Maxwell equations (\ref{closed}--\ref{energy}) on $Y$ 
 iff the following conditions all hold there:
\begin{eqnarray}
dF^+&=&0 \label{self}\\
s&=&\mbox{\rm const}\label{constant}\\
\ro&=& - 2 F^+\circ F^- ~.\label{matter}
\end{eqnarray}
\end{prop}
\begin{proof} We begin by proving  \eqref{leibniz}, on which the rest of the argument hinges. Using  lower-case latin letters for vector indices, and upper-case
primed and unprimed latin letters for the corresponding spinor indices, we have 
\begin{eqnarray*}
\nabla^a\left( {{{F^+}_a}^c}{F^-}_{cb}\right) &=&\nabla^{AA'} \left( {{F^+}_A}^C{\varepsilon_{A'}}^{C'}{F^-}_{C'B'}\varepsilon_{CB}\right) \\
&=& \nabla^{AA'} \left( {F^+}_{AB} {F^-}_{A'B'}\right) \\
&=&   F^-_{A'B'}\nabla^{AA'} F^+_{AB}+F^+_{AB} \nabla^{AA'} F^-_{A'B'}\\
&=&  F^-_{D'B'}{\varepsilon_{A'}}^{D'}\nabla^{AA'} \left( \varepsilon_{DB}{F^+_A}^D\right)+ F^+_{DB}{\varepsilon_{A}}^{D} \nabla^{AA'}\left(  \varepsilon_{D'B'} {F^-_{A'}}^{D'} \right)\\
&=&  F^-_{D'B'}\varepsilon_{DB}\nabla^{AA'} \left( {F^+_A}^D{\varepsilon_{A'}}^{D'}\right)+ F^+_{DB} \varepsilon_{D'B'}\nabla^{AA'} \left(  {F^-_{A'}}^{D'}{\varepsilon_{A}}^{D}\right)\\
&=&  F^-_{db}\nabla^{a} {F^+_a}^d+ F^+_{db} \nabla^{a}  {F^-_a}^d~,\\
\end{eqnarray*}
so that 
$$\nabla\cdot (F^+ \circ F^-) = F^- (\nabla\cdot F^+)  + F^+  (\nabla \cdot F^-),$$
which is exactly the desired identity \eqref{leibniz}. 

The endomorphisms of the tangent bundle corresponding to  self-dual or anti-self-dual forms
of length $\sqrt{2}$  are almost-complex structures;  those arising from $\Lambda^+$ moreover commute
with those arising from $\Lambda^-$, and the composition of any such pair of almost-complex structures is  trace-free. Thus $F^+\circ F^+$ and $F^-\circ F^-$ are multiples of $h$,
and $F^+\circ F^-= F^-\circ F^+$. Hence $[F\circ F]_0= 2 F^+\circ F^-$, and \eqref{energy} is therefore
algebraically equivalent to \eqref{matter}.  In particular, either  \eqref{energy} or  \eqref{matter} implies
$$
-{\textstyle\frac{1}{2}} \nabla\cdot \ro = F^- (\nabla\cdot F^+)  + F^+  (\nabla \cdot F^-)
$$
via \eqref{leibniz}, and so, by the doubly-contracted Bianchi identity, implies that 
$$
-{\textstyle\frac{1}{8}} ~ds = F^- (\nabla\cdot F^+)  + F^+  (\nabla \cdot F^-)~.
$$
In particular, (\ref{closed}--\ref{energy}) $\Longrightarrow$  (\ref{self}--\ref{matter})  on any open set. 
Conversely, (\ref{self}--\ref{matter}) imply that 
$$0 = F^+  (\nabla \cdot F^-),$$
so that, on the open set $Y$ where $F^+$ is invertible, the anti-self-dual $2$-form $F^-$ 
is co-closed, and  $F=F^++F^-$ is therefore  both closed and co-closed. Thus, on $Y$, 
  (\ref{self}--\ref{matter}) $\Longleftrightarrow$ (\ref{closed}--\ref{energy}), as claimed. 
\end{proof}

Our next result depends on the notion of a {\em holomorphy potential}. This concept, already 
 implicit in the work of  Matsushima \cite{mats}  and Lichnerowicz \cite{licmat}, was  eventually 
codified  by Calabi and others \cite{bes,calabix2}. A complex valued function $f: M\to \CC$
on a
K\"ahler manifold $(M,g,J)$ is called a {\em holomorphy potential}  if the $(1,0)$ component of its gradient is 
a holomorphic vector field:
$$\nabla_{\bar{\mu}}\nabla^{\nu}f =0.$$
By lowering an index, this is equivalent to saying that the $\odot^2 \Lambda^{0,1}$  
component of 
its Hessian  vanishes: 
$$\nabla_{\bar{\mu}}\nabla_{\bar{\nu}}f=0.$$
In the special case that $f$ is {\em real}, this is equivalent to saying that its Hessian belongs to 
$\Lambda^{1,0}\otimes \Lambda^{0,1}$, or in other words that $\Hess f$ is $J$-invariant. 
In this real case, something even more remarkable happens. First of all, $\xi= J \nabla f$
is the imaginary part of a holomorphic vector field, and so the flow of $\xi$ preserves $J$. 
At the same time, $\xi$ is the symplectic gradient of a function, so its flow also preserves the K\"ahler
 form $\omega$.
But since $g=\omega(\cdot , J \cdot )$,  this implies that the flow of $\xi$ also preserves $g$.
In other words, $\xi = J\nabla f$ is a Killing field whenever $f$ is a real holomorphy potential,
and this assumption on a real-valued function $f$ is equivalent to requiring that $\Hess f$ be a
$J$-invariant symmetric tensor.

\begin{prop} \label{tag} 
Let $(h,F)$ be a strongly Hermitian solution of the Einstein-Maxwell equations on 
a  complex surface $(M^4,J)$. Let $f=2^{-1/4}|F^+|^{1/2}$, and let $Y$ be the 
(possibly empty) open set where $f\neq 0$. Then $g=f^2h$ is a K\"ahler metric
on $Y$, and $f$ is a real holomorphy potential on $(Y,g)$, possessing   the additional property that 
$h=f^{-2}g$ has constant scalar curvature.  
\end{prop}
\begin{proof}
On a Hermitian surface, the $J$-invariant $2$-forms are given by 
$$\Lambda^{1,1}= \RR \omega_h \oplus \Lambda^-~,$$ 
so the self-dual part $F^+$ of any $J$-invariant $2$-form $F$ is a function times the 
associated $2$-form $\omega_h= h(J\cdot , \cdot )$ of our metric. If, near any point 
where $F^+$  is non-zero, we rescale our Hermitian metric
so as to give it pointwise norm $\sqrt{2}$, then,  
the self-dual $2$-form $\pm F^+$ is the associated $2$-form of
a new metric $g$ which is conformal to $h$;  and if $F^+$ is closed, the resulting metric $g$ is then 
K\"ahler. This is exactly realized by setting $g=f^2h$. 

On the other hand, if $F^+$ is a multiple of $\omega_h$, $F^+\circ F^-$ is a multiple of
$F^-(J\cdot, \cdot )$, so, given that $(h,F)$ is a strongly Hermitian solution of the Einstein-Maxwell
equations, \eqref{matter} implies that $\ro_h$  is $J$-invariant. However, 
on the set where $f\neq 0$, the function  $f$ is smooth and $g$ is defined, the trace-free Ricci tensors of the
two metrics are related \cite{bes} by 
\begin{equation}
\label{transric}
\ro_h=\ro_g + 2f^{-1} \Hess_0f
\end{equation}
where the trace-free Hessian of $f$ is computed with respect to $g$. 
Since both $g$ and $h$ have $J$-invariant Ricci tensors, it follows that 
$\nabla \nabla f$ is $J$-invariant. Hence $f$ is a real positive holomorphy potential 
on $(Y,g)$, and $h=f^{-2}g$ has constant scalar curvature by \eqref{constant}.
\end{proof}

While we have arranged to make $\pm F^+$ into the K\"ahler form $\omega_g$ of
the conformally related K\"ahler metric $g$, the reader might be right to worry
that the sign might be different on different connected components of $Y$. However, 
assuming that  $M$ is connected, this turns out to be impossible. Indeed, by a theorem 
of B\"ar \cite{baer}, the zero locus of the closed and co-closed form $F^+$ has
Hausdorff dimension $\leq 2$, and,  as we will show in  of \S \ref{scenario}, Lemma \ref{path}, this means it cannot
disconnect $M$. Since the Einstein-Maxwell equations are invariant under $F\to -F$,
we will therefore eventually be entitled to  assume that $\omega=+F^+$, at the modest price of perhaps
changing the  sign of $F$.

We conclude this section with a  partial converse of the above result:

\begin{prop}
Let  $f: M \to \RR^+$ be a positive  holomorphy potential on a K\"ahler surface $(M^4,g,J)$
with K\"ahler form  $\omega= g(J\cdot , \cdot)$.
If $h=f^{-2}g$ has constant scalar curvature, then there  is a unique $2$-form $F$ with 
$F^+=\omega$ such that $(h, F)$ is a strongly Hermitian solution of the Einstein-Maxwell equations. 
\end{prop}
\begin{proof}
Since $f$ is a positive function with 
 $J$-invariant Hessian, \eqref{transric} guarantees that the trace-free Ricci curvature
of $h=f^{-2}g$ is $J$-invariant, and so can be uniquely written as $\ro_h=\varphi (\cdot , J\cdot )$
for a unique $\varphi\in \Lambda^-$. Setting $F^+=\omega$ and $F^-= \frac{1}{2}f^{-2}\varphi$ 
  then produces a solution of \eqref{self} and \eqref{matter}, and this choice of 
  $F^-\in \Lambda^-$   is  moreover
the only one 
that satisfies \eqref{matter} in conjunction with $F^+=\omega$. This ansatz thus solves (\ref{self}--\ref{matter})
 iff \eqref{constant} is satisfied, and therefore solves   (\ref{closed}--\ref{energy}) iff $h$ has constant scalar curvature.
\end{proof}

\section{The First Main Theorems}
\label{scenario}

The main worry aroused by the results in the previous section is that the construction 
breaks down at the zero locus of the $2$-form $F^+$. Fortunately, however, this worry turns
out to be largely misplaced. To get around this problem, we will develop a sequence of lemmata 
inspired by an argument of Derdzi\'nski \cite{derd}, now carefully implemented by 
using   a fundamental  result on zero sets of
generalized harmonic spinors due to B\"ar \cite{baer}. To do the job properly, 
we begin with a regularity result: 

\begin{prop} \label{regular} 
Let $(h,F)$ be a solution of the Einstein-Maxwell equations (\ref{closed}--\ref{energy}) 
on a smooth $4$-manifold $M$.
Suppose that,  in some coordinate atlas,  $h$ is of class $C^{2,\alpha}$ for some $\alpha > 0$,  and that $F$ is of class $C^1$. 
Then $h$ and $F$ are $C^\infty$ in harmonic coordinates. 
\end{prop}
\begin{proof} By the results of DeTurck and Kazdan \cite{detkaz}, we can pass to harmonic coordinates
without losing any regularity, and if,  in harmonic coordinates, $h$ is  of class $C^{k,\alpha}$, 
with  Ricci tensor $r$ of  class $C^{k,\alpha}$, then $h$ is actually of class $C^{k+2,\alpha}$. On the other hand, since $F$ is in the kernel of $d+d^*$, elliptic regularity \cite{morrey} implies that if $h$ is of class $C^{k,\alpha}$, 
then $F$ is of class $C^{k,\alpha}$, too. However, equation \eqref{energy} tells us that 
$r=-[F\circ F]_0+ \lambda h$ for some constant $\lambda$. Thus, if $h$ is of class $C^{k,\alpha}$,
$r$ is also of class $C^{k,\alpha}$, so $h$ is actually $C^{k+2,\alpha}$. It therefore follows by induction (``bootstrapping'') that
$h$ and $F$ are both smooth in harmonic coordinates.   \end{proof}

In particular, the implicit assumption of smoothness used throughout \S \ref{overture} can now be seen to have been perfectly justified. 

Next, we will  establish a technical lemma of key importance. 

\begin{lem} \label{extend} 
Let $(X,h)$ be a  $C^3$  Riemannian $n$-manifold, and let $Z\subset X$ be a closed subset of 
 Hausdorff dimension $< (n-1)$. Let $Y=X-Z$ be the complement of $Z$, and suppose that $\xi$ is a Killing field on $(Y,h)$. Then  $\xi$  extends to $(X,h)$ as a Killing vector field $\hat{\xi}$.
 \end{lem}
 \begin{proof}
 Let $q\in Z$ be any point, and let $U\subset M$ be a geodesically convex neighborhood of 
 $q$. Since $Z$ has $n$-dimensional measure zero, it has empty interior, and it follows that $U-Z$
 is non-empty. Thus, there exists some $p\in U$ which belongs to the complement $Y$ of $Z$;
 and since $Z$ is  closed, some small metric ball $B_{2\varepsilon} (p)$ is also contained in $U-Z$. 
 Let $S\approx S^{n-1}$ be the unit sphere in $T_pM$, and let $\Pi : (U-\{ p\}) \to S$ be the 
 $C^2$  map which sends  $x\in (U-\{ p\})$ to the initial unit tangent vector  of  the geodesic segment 
 $\overline{px}$. Since $\Pi$ is Lipschitz, the Hausdorff dimension of $\Pi ( Z\cap U)$ is also 
 $< (n-1)$, so almost every geodesic through $p$ misses $Z$. Now recall that the restriction of
 a Killing field $\xi$ to any geodesic $\gamma$ solves  Jacobi's equation, since the flow of $\xi$
 sends $\gamma$ to a family of geodesics. Let  us therefore define a $C^1$ vector field 
 $\hat{\xi}$ on $U-\{p\}$, as the unique family of solutions of Jacobi's equation
 along  geodesics radiating from $p$, with the same  initial values and initial (radial) derivatives as $\xi$ 
 along the sphere $S_\varepsilon ( p) =
 \partial B_\varepsilon (p)$. 
 Then $\xi$ and $\hat{\xi}$ are both $C^1$ on $U-Z-\{p\}$, and agree on 
 $W$, where $W\subset (U-Z-\{ p\})$ is 
 the union  of all the geodesics
 radiating from $p$ which miss $X$. However, $W$ is 
 dense in $U-Z-\{p\}$, since its complement $\Pi^{-1}[\Pi (Z\cap U)]$ is of $n$-dimensional measure zero. 
  Because  two continuous vector fields 
 on $U-Z-\{p\}$ which agree on a dense set must be equal, it follows that $\hat{\xi}=\xi$ on $U-Z-\{ p\}$.
 We can therefore  extend $\xi$ across $Z$ as $\hat{\xi}$, and  extend $\hat{\xi}$ 
 across $p$ as $\xi$. Moreover, since the $C^1$ vector field $\hat{\xi}$ solves Killing's equation on the 
open  dense set $U-Z-\{ p\}$ where it coincides with $\xi$, it actually solves Killing's equation everywhere. 
Finally, since  the intersection of two geodesically convex sets is geodesically convex, any two such local extension of $\xi$ across $Z$ agree on the overlap, and we can therefore consistently extend  $\xi$ 
to $M$ as a Killing field $\hat{\xi}$. 
 \end{proof}
 
 The above argument also establishes a minor noteworthy point: 
 
 \begin{lem}\label{path}
 Let $X$  be a smooth connected  $n$-manifold, and let $Z\subset X$ be a
 closed subset of Hausdorff dimension $< (n-1)$. Then $X-Z$ is path connected. Moreover, 
 for any  metric $h$ on $X$, the 
 Riemannian distance in  $(X-Z,h)$ is just the restriction of the Riemannian distance from 
 $(X,h)$.
 \end{lem}
 \begin{proof} It suffices to prove to prove the statement when $h$ is smooth, as the general case  then follows by 
 quasi-isometric pinching. 
  In any small geodesically convex ball $U$, a dense set
 of  points of $U-Z$  can be reached from any  $p\in U-Z$ by following 
 distance-minimizing geodesics  that avoid $Z$. We can 
 therefore reach any $q\in U-Z$ from $p\in U-Z$ by following  broken geodesic paths $\overline{pq_j}\cup \overline{q_jq}$ in $U-Z$, where $q_j\to q$, and, as $j\to \infty$, 
  the lengths of such  paths approach the Riemannian distance from $p$ to $q$ in $X$. Any piecewise geodesic path 
 in $X$ joining two points in $X-Z$ can therefore be approximated by piecewise geodesic paths in $X-Z$ 
 of
 essentially the same length.
 \end{proof}
 
 \begin{lem} \label{tip} 
 Let $(M^n,J, h)$, $n=2m\geq 4$,  be a  Hermitian manifold with \linebreak 
 $J$-invariant Ricci tensor,
and suppose that 
 $f$ is a continuous non-negative function such that $g=f^2h$ is a  K\"ahler metric 
 on the open subset $Y$ where $f$ is non-zero.   Relative to the  complex coordinate atlas, assume that $f$ is at least $C^2$ on $Y$, 
 and that $h$ is at least $C^3$ in some $C^2$-compatible coordinate atlas for $M$.  
 Suppose, moreover, that 
 $Z:=f^{-1}(0) = M-Y$  has Hausdorff dimension $< (n-1)$. Then $Z=\varnothing$, 
 $f$ is everywhere positive, and $(M,J,h)$ is globally conformally K\"ahler. 
 \end{lem}
  \begin{proof}
  The trace-free Ricci tensors of the conformally related metrics $h$ and $g$ are related \cite{bes} by 
  $$\ro_h=\ro_g + (n-2) f^{-1}\Hess_0 f$$
  where the trace-free Hessian of $f$ is computed with respect to $g$. Since
  the Ricci tensors of both metrics are $J$-invariant, it follows that the component of
  $\Hess f$ in $\odot^2 \Lambda^{0,1}$ vanishes, so that 
  $f$ is a real holomorphy potential on $(Y,g)$, and its 
   symplectic gradient $\xi= J \grad_g f$ is therefore a Killing field.
  However, $\xi f=0$, so the flow of $\xi$ preserves not only $g$, but also $h=f^{-2}g$; that is, 
   $\xi$ is a Killing field for $h$, defined on the complement of a closed 
  set $Z$ of Hausdorff dimension $< (n-1)$. 
    Thus $\xi$ extends across $Z$ as a Killing field by Lemma \ref{extend}, and 
  the vector field $-J\xi = \grad_g f$ therefore extends across $Z$, too. However, 
  $\grad_g f=f^{-2}\grad_h f= -\grad_h f^{-1}$, so, by lowering an index, 
   it follows that the  $1$-form $\phi=h(J\xi, \cdot )= d(f^{-1})$
  extends across $Z$ as a  $1$-form $\hat{\phi}=h(J\hat{\xi}, \cdot )$ which is at least $C^1$.  
  On the other hand, 
   the $1$-form $\phi$ satisfies $d\phi=0$ on open the dense set $Y\subset M$,
  so $d\hat{\phi}=0$. If $U$ is a geodesically 
  convex neighborhood of some $q\in Z \subset M$, the Poincar\'e lemma
  therefore tells us that  $\hat{\phi}= du$ for
  some $C^2$ function $u$ on $U$, since $U$ is contractible. On the other hand, 
  Lemma \ref{path} guarantees that  $U-Z$ is  path connected. However, $d(u-f^{-1})=\hat{\phi}-\phi=0$ on $U -Z$, and since $U-Z$ is
  connected, it follows that $u - f^{-1}$ is constant. Thus $f^{-1}$ 
  extends across $q\in Z$ as $u+\mbox{const}$. However, since $f(q)=0$, this is a contradiction.
  We are therefore forced to conclude that $Z = \varnothing$, and that $f> 0$ on all of $M$. 
  In particular, $h=f^{-2}g$ is globally conformally K\"ahler. 
\end{proof}

We are now ready to prove our first main results. 

\begin{thm} Let $(h,F)$ be a strongly Hermitian solution of the Einstein-Maxwell equations on 
a complex surface $(M^4,J)$, and assume   that $h$ is {\em not} Einstein. Also, for some $\alpha > 0$,  suppose that 
$h$ is of differentiability class  
$C^{2,\alpha}$ 
  and  that $F$ is at least  $C^1$, relative to the complex atlas of $(M,J)$.
Then there is a $C^\infty$  K\"ahler metric $g$ on $(M,J)$ and a $C^\infty$  positive holomorphy potential
$f$ on $(M,J,g)$ such that $h=f^{-2}g$, and such that  $F^+$ is a constant times the K\"ahler form $\omega$ of $g$.
\end{thm}
\begin{proof}
First notice that $h$ is smooth in harmonic coordinates by Proposition \ref{regular},
and we can therefore invoke the work of B\"ar \cite{baer}.  The self-dual $2$-form 
$F^+$ is in the kernel of the Dirac-type operator $d+d^*$ on $(M,h)$, and cannot be 
identically zero, because otherwise \eqref{matter} would force $h$ to be Einstein, contrary
to our assumptions. Thus B\"ar's theorem \cite{baer} asserts that the zero locus $Z$
of $F^+$ has Hausdorff codimension $\geq 2$. It follows that $f=2^{-1/4}|F^+|^{1/2}$ 
and $g=f^2h$ 
 fulfill the hypotheses of Lemma \ref{tip}, since 
 elliptic regularity guarantees that 
$F$ is at least $C^{2,\alpha}$ in complex coordinates. Hence $Z=f^{-1}(0)$ is empty, and  $(M,h)$ is globally conformally 
K\"ahler, with $h=f^{-2}g$ for a K\"ahler metric $g$ and a positive holomorphy potential $f$. 

It only remains to show that $g$ and $f$ are actually smooth with respect to  the complex coordinate
atlas of $(M,J)$. To this end, first notice that $h$ and $f= 2^{-1/4}|F^+|^{1/2}\neq 0$ are 
smooth in the harmonic coordinate atlas for $h$ by Proposition \ref{regular}. 
It follows that $g=f^2h$ is smooth in these coordinates, too. Elliptic regularity thus implies that 
any function that is harmonic with respect to $g$ is also  smooth with respect to 
harmonic coordinates for $h$. However, the real 
and imaginary parts of any local holomorphic function are harmonic with respect to 
any K\"ahler metric. Since $g$ is K\"ahler, to follows that the complex coordinate atlas
of $(M,J)$ defines the same $C^\infty$ structure as the harmonic atlas of $h$. 
\end{proof}

To prove  Theorems \ref{clef} and \ref{alto}, it thus suffices to address the case in which  
$h$ is Einstein. However, this case follows \cite{lebhem}
from  the work of Derdzi\'nski \cite{derd}, together with 
  the Riemannian Goldberg-Sachs theorem \cite{aggs,nur} and results of Boyer \cite{boy2}
 on  compact anti-self-dual Hermitian manifolds. Here one of Derdzi\'nski's important discoveries is
 that  $W_+$ is either nowhere zero or must vanish identically. We note in passing that Lemma
 \ref{tip},  with $f=(24|W_+|^2)^{1/3}$, also  provides a clear and watertight way of establishing this point. 
Indeed, the equation $\nabla\cdot W_+=0$ is also of generalized Dirac type, so 
B\"ar's theorem once again implies that the zero locus of $W_+$ has Hausdorff codimension $\geq 2$
on any non-anti-self-dual Einstein $4$-manifold.

\section{Some Compact Examples}
\label{construct}

Let us now construct some compact, non-K\"ahler examples of strongly Hermitian solutions of the 
Einstein-Maxwell equations. To do this, we will 
look for K\"ahler metrics $g$ on a product $\CP_1 \times \Sigma$, together with  
positive holomorphy potentials $f$ such that $h=f^{-2}g$ has constant scalar curvature.
For simplicity, we will take $g$ to be the Riemannian product of an axisymmetric metric
on $S^2=\CP_1$ with a metric of constant scalar curvature $s_2=\mathbf{c}$ on $\Sigma$, and we will
take our holomorphy potential $f$ to be the Hamiltonian for  rotation of $S^2$ about the given
axis, with period $2\pi$. Thus, 
$$
g=g_1 + g_2
$$
where $(\Sigma, g_2)$ has constant scalar curvature $\mathbf{c}\in \RR$, and where
the metric $g_1$ on $S^2$ can be written in cylindrical coordinates $(t, \theta) \in (a,b) \times (0,2\pi ]$
as 
$$g_1 = \frac{dt^2}{\Psi(t)} + \Psi(t) d\theta^2$$
for some smooth positive function $\Psi(t)$. Here we have put the K\"ahler form 
$$\omega_1 = dt\wedge d\theta$$
in Darboux coordinates,  so that we may assume that our holomorphy potential 
is given by $f=t$ as long as we remember to insist  that $b> a> 0$. The scalar curvature of $g$ is then given by
$$s_1 = \Delta_{g_1}\log \Psi = - \Psi^{\prime\prime}(t).$$
Thus the scalar curvature of $g$ is given by 
$$s=s_1+s_2 = \mathbf{c} - \Psi^{\prime\prime}(t).$$
We now want to arrange that $h=f^{-2}g$ has constant scalar curvature, which is to say that 
\begin{equation}
\label{yada}
(6\Delta_g + s)f^{-1} =  f^{-3}\mathbf{d}
\end{equation}
for a constant $\mathbf{d} =s_h$. We may now rewrite this as
$$s =  f^{-2}\mathbf{d} - 6f\Delta f^{-1}$$
or in other words, as 
$$\mathbf{c} - \Psi^{\prime\prime}= \frac{\mathbf{d}}{t^2} - 6t \Delta_{g_1} \left(\frac{1}{t}\right)$$
since the Hessian of $f=t$ is trivial in the $\Sigma$-directions of our Riemannian product. 
Since 
$$\Delta_{g_1} \left(\frac{1}{t}\right) = \left( \frac{\Psi}{t^2}\right)^\prime ,$$
the Yamabe equation \eqref{yada} therefore reduces to the ODE
$$
\mathbf{c} - \Psi^{\prime\prime}= \frac{\mathbf{d}}{t^2} - 6 \frac{\Psi^\prime}{t} + 12 \frac{\Psi}{t^2} , 
$$
or equivalently 
\begin{equation}
\label{keats}
t^2 \Psi^{\prime\prime} - 6t \Psi^{\prime} + 12 \Psi =  \mathbf{c} t^2 -\mathbf{d}.
\end{equation}
Since the linear operator 
$$y\longmapsto  t^2y^{\prime\prime} - 6 ty^\prime + 12 y$$
acts on monomials by 
$$t^n \longmapsto (n-3)(n-4) t^n$$
the general solution of equation \eqref{keats}  is therefore a quartic polynomial 
$$\Psi(t) = At^4 + Bt^3 + \frac{\mathbf{c}}{2}t^2 - \frac{\mathbf{d}}{12}$$
with $\Psi^\prime(0)=0$. Now, in order to get a metric on $S^2$, we need to impose the 
boundary conditions that 
$$\Psi(a) = \Psi(b) = 0, \quad \Psi^{\prime} (a) = - \Psi^\prime (b) = 2,$$
while remembering that we also must have $\Psi^{\prime} (0)=0$ and $\Psi(t)> 0$ on $(a,b)$; for example,
if we set $t-a=\mathbf{r}^2/2$, these conditions guarantee that the metric takes the form 
$(1+O(\mathbf{r}^2))d\mathbf{r}^2+ (1+ O(\mathbf{r}^2))\mathbf{r}^2 d\theta^2$ for 
small positive $t-a$. 
For $0< a < b$, the unique such quartic polynomial is given by 
\begin{center}
\includegraphics[scale=.2]{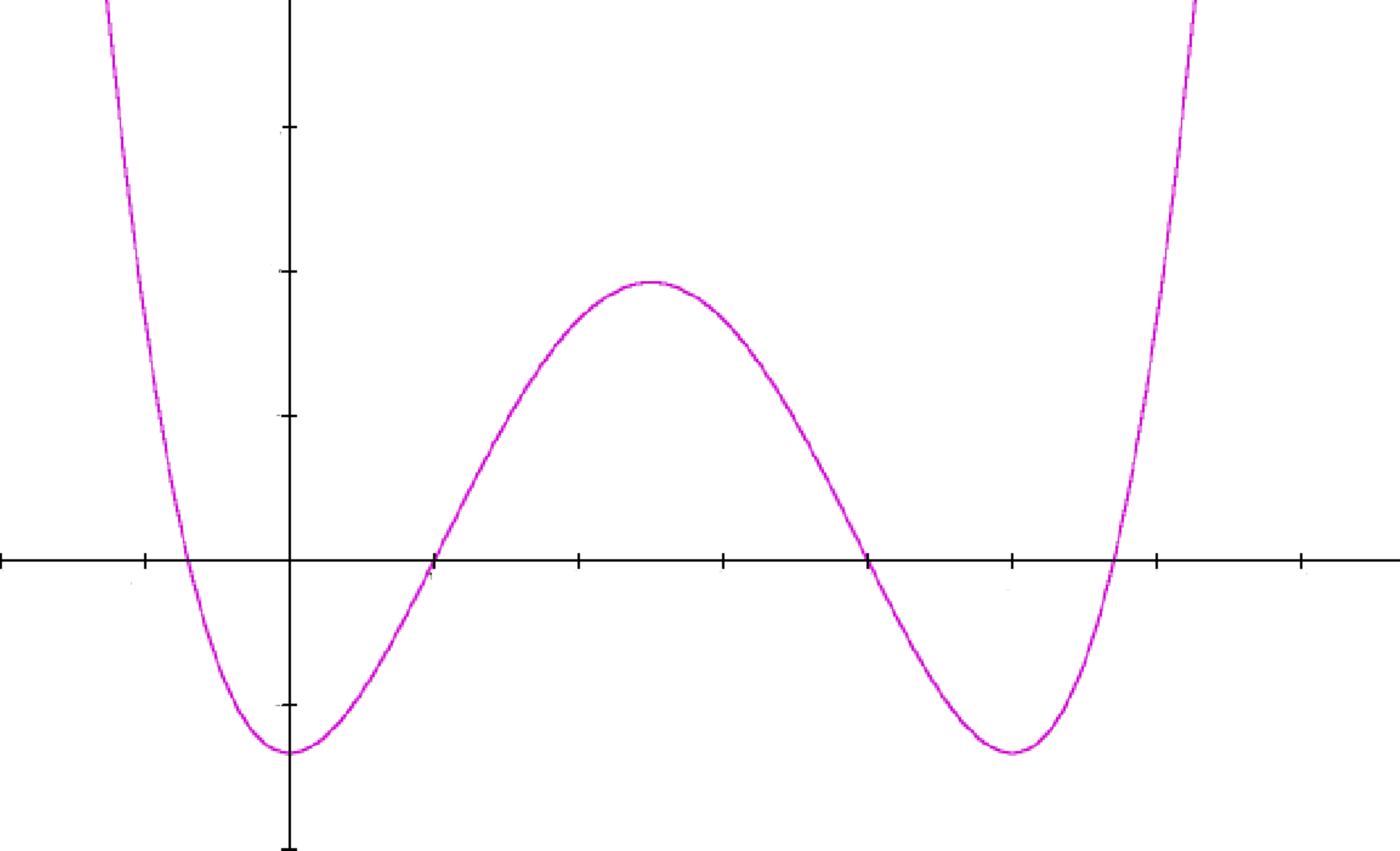}
\end{center}
\begin{equation}
\label{solution}
\Psi (t)= \frac{(t-a)(t-b)}{a-b}\left[ 2 - \frac{(t-a)(t-b)}{ab}\right] ~. 
\end{equation}
One can then read off that 
$$\mathbf{d}= -12 ~\Psi(0)= \frac{12ab}{b-a} > 0$$
and that 
$$
\mathbf{c} = \Psi^{\prime\prime}(0)= \frac{2(a+b)^2}{(b-a)ab}> 0.
$$
In particular, the constant scalar curvature $s_2=\mathbf{c}$ of $\Sigma$ must be positive, 
so $\Sigma$ must be a $2$-sphere $\CP_1$.  Gauss-Bonnet moreover guarantees that the total area 
of $\Sigma$ must be 
$$\omega (\Sigma) = \frac{4\pi (b-a) ab}{(a+b)^2}.$$
By contrast, the total area of $(\CP_1, g_1)$ is $4\pi (b-a)$. Thus, we have constructed
a K\"ahler metric 
$$g=g_1+g_2$$
on $\CP_1\times \CP_1$
which is not locally symmetric, but engenders a solution 
\begin{eqnarray*}
h &=& \frac{1}{t^2}g\\
F^+&=&\omega
\end{eqnarray*}
of positive scalar curvature $\mathbf{d}=s_h$, 
and belongs to the K\"ahler class
$$[\omega ] = 4\pi (b-a) \left( \mathfrak{S}_2+
\frac{b/a}{(1+b/a)^2}\mathfrak{S}_1 \right)\in H^2(\CP_1\times \CP_1),$$
where $\mathfrak{S}_1$ and $\mathfrak{S}_2$ are the Poincar\'e duals of the first and second factors
of $\CP_1\times \CP_1$. As we allow $b> a> 0$ to vary, these sweep out all
the K\"ahler classes for which the second factor has  less than one-quarter the area  of the first factor. 
In particular, the constant-scalar-curvature K\"ahler metric in such a class $[\omega]$, obtained by 
taking the Riemannian product of round $2$-spheres of appropriate radii, is not the only 
K\"ahler metric in $[\omega]$ that engenders a solution of the Einstein-Maxwell equations. 
The solutions $(h,F)$ so engendered are moreover geometrically distinct; one family  of solutions
consists of  symmetric spaces, 
whereas the metrics in the other family are not even locally symmetric.  Theorem \ref{behold}
now follows. 

\section{Concluding Remarks}
\label{conclusion}

We have just seen that not every strongly Hermitian Einstein-Maxwell solution on a compact
complex surface is given by a cscK metric. This presents us with the intriguing problem of
determining when such non-K\"ahler solutions exist. One hint is provided by the following easy result:

\begin{prop}\label{rule} 
Let  $(h,F)$ be a strongly Hermitian solution of the Einstein-Maxwell equations 
(\ref{closed}--\ref{energy}) on a compact complex surface $(M^4,J)$. Then 
either $h$ is a constant-scalar-curvature K\"ahler metric, or else
$(M,J)$ is a rational or ruled surface. 
\end{prop}
\begin{proof}
If the holomorphy potential $f$ is non-constant, $\xi =J \nabla f$ is a non-trivial Killing field
for $g$. The flow of $\xi$ is then a connected  Abelian group of isometries of $(M,g)$, the closure of 
which is a torus subgroup of the isometry group which acts on $M$ with non-empty fixed-point set. 
However, the generators are also the real parts of holomorphic vector fields. 
If this torus has dimension $> 1$, one gets a non-trivial holomorphic section of  $K^{-1}$ with non-empty 
zero locus by taking the wedge product of two independent holomorphic vector fields associated with the action. 
 Otherwise, 
$\xi$ is a periodic vector field, and we obtain an embedded rational curve of non-negative self-intersection 
by taking the closure of generic orbit of the group generated by $\xi$ and $\nabla f = -J\xi$. 
Either way, it follows that the plurigenra $p_\ell (M,J)= h^0 (M, \mathcal{O} (K^\ell ))$, $\ell \in \mathbb{N}$, must all vanish, and, since $(M,J)$ also admits a K\"ahler metric $g$, 
 surface classification \cite{bpv,GH} tells us $(M,J)$ is rational or ruled.\end{proof}

Of course, the  Einstein Hermitian metrics \cite{chenlebweb,derd,page} 
on $\CP_2\# \overline{\CP}_2$ and  $\CP_2\#2 \overline{\CP}_2$ provide two more examples, so it
seems certain that the full story will turn out to be rich and interesting. But there is obviously an
enormous 
gulf between the minuscule menagerie of currently known examples and the world of possibilities 
allowed by Proposition \ref{rule}. I can only hope that some interested reader will feel motivated to 
construct some further examples!

The examples constructed in \S \ref{construct} have an intriguing feature that is also worth mentioning here. The constructed metrics $h$ of course all have constant scalar curvature. The question is, 
which of them, if any, are {\em Yamabe metrics}. Recall that Yamabe's program for proving the existence
of constant-scalar-curvature metrics on compact manifolds involves {\em minimizing} the functional \eqref{hilbert} (or the appropriate $n$-dimensional generalization) in any
given conformal class; such metrics always exist \cite{lp,rick}, and are called Yamabe metrics. 
But while every Yamabe metric has constant scalar curvature, not every constant-scalar-curvature
metric is Yamabe. This complication only occurs when the scalar curvature is positive, as is the case 
 for
the examples in question. For precisely this reason, the  {\em Yamabe invariant} of 
$4$-manifolds like $S^2 \times S^2$ remains unknown. Here, the Yamabe invariant $\mathcal{Y}(M)$ 
of a $4$-manifold is obtained  by  
by taking the infimum of \eqref{hilbert} in each conformal class, and then taking the supremum of these infima 
over all conformal classes; equivalently, it is the supremum of the scalar curvatures of all unit-volume
Yamabe metrics on $M$. What is currently known \cite{bwz,okob,lcp2,gl1} strongly 
suggests  that one should have 
$$12\pi\sqrt{2} = \mathcal{Y}(\CP_2)\leq \mathcal{Y}(S^2 \times S^2 )\leq \mathcal{Y}(S^4 )= 8\pi \sqrt{6},$$
but here, for the moment,  the lower bound is merely conjectural. On the other hand, the value 
of the functional \eqref{hilbert} is easy to compute for the constructed metrics $h=g/t^2$
on $\CP_1\times \CP_1$; namely, since the scalar curvature $s_h$ is exactly the constant 
$\mathbf{d}$, one can easily show that
$$s_hV^{1/2} = 8\pi  \frac{\sqrt{6(a^2+ab+b^2)}}{a+b}~.$$
Intriguingly, as $b/a$ ranges over the interval $(1,\infty)$, this exactly sweeps out the interval 
$(12\pi \sqrt{2}, 8\pi \sqrt{6})$ in question. 
Thus,  if any of these metrics is Yamabe, the conjectural lower bound would be established;
and if they are {\em all} Yamabe, the upper bound would be saturated! 

\begin{ack}
The author would like to warmly thank Caner Koca for    conversations which helped stimulate the present piece of research, and 
the anonymous referee for a wonderfully useful  and   detailed   list of suggestions for minor improvements to the article. 
\end{ack}

\end{document}